\documentclass[12pt]{article}

\pdfoutput=1
\usepackage{graphicx}
\usepackage{color}
\usepackage{hyperref}
\usepackage{url}
\usepackage{multirow}
\usepackage{enumerate}
\usepackage{amsmath, amsthm, amssymb}
\usepackage{mathtools}
\usepackage{xspace}

\newtheorem{defn}{Definition}

\newtheorem{lem}{Lemma}
\newtheorem{theo}{Theorem}
\newcommand{\smodtitle}{mod\textsuperscript{$\star$}\xspace} 
\newcommand{\smodtext}{mod\textsuperscript{$\star$}\xspace} 
\newcommand{\smod}{\textnormal{mod}^\star}
\newcommand{\smodn}[1]{\;\;(\textnormal{mod}^\star~#1)}

\begin{document}
\title{Modified Congruence Modulo $n$ with Half the Amount of Residues}

\author{
    Gerold Br\"andli\\
    Schanzm\"attelistrasse 27\\
    5000 Aarau\\
    Switzerland\\
    braendli@hispeed.ch\\
  \and
    Tim Beyne\\
    Rotspoelstraat 15\\
    3001 Heverlee-Leuven\\
    Belgium\\
    tim.beyne@student.kuleuven.be
}

\maketitle
\begin{abstract}
    We define a new congruence relation on the set of integers, leading to a group similar to the multiplicative group of integers modulo $n$.
    It makes use of a symmetry almost omnipresent in modular multiplications and halves the number of residue classes.
    Using it, we are able to give an elegant description of some results due to Carl Schick, others are reduced to well-known theorems from algebra and number theory.
    Many concepts from number theory such as quadratic residues and primitive roots are equally applicable. It brings noticeable advantages in studying powers of odd primes, and in particular when studying semiprimes composed of a pair of related primes, e.g.\ a pair of twin primes. Artin's primitive root conjecture can be formulated in the new context. Trigonometric polynomials based on chords and related to the new congruence relation lead to new insights into the minimal polynomials of $2\cos(2\pi/n)$ and their relation to cyclotomic polynomials. 
\end{abstract}

\section{Introduction}
The motivation for this paper comes from the work of Carl Schick \cite{Sch03, Sch08, Sch13}.
In 2003, Schick found a recurrence relation \cite{Sch03} that yields, for every odd natural number $n$, a characteristic cyclic sequence of positive and negative odd integers.
\par The terms $(q_i)_{i \in \mathbb{N}}$ of this sequence are given by:
\begin{equation}
\label{eq:schickseq1}
    q_i = n - 2|q_{i - 1}|\; \text{with} \;q_1=(-1)^{(n+1)/2}.
\end{equation}
The following modified recurrence relation yields the absolute value of these terms:
\begin{equation}
\label{eq:schickseq2}
    q_i = |n - 2q_{i - 1}|\; \text{with} \;q_1=1.
\end{equation}
\par In this paper we interpret his findings in a broader context. By introducing a new congruence relation, denoted $\smod$, new insights into the work of Schick and others are gained.
\par Whereas the standard congruence relation (``mod $n$'') yields a least residue system that can be represented by even and odd nonnegative integers smaller than $n$, the proposed relation leads to a system whose elements can be represented e.g.\ by retaining only odd representatives. 
Hence, the size of the residue system is exactly halved.
Well known concepts from number theory such as ``quadratic residue'', ``multiplicative group'' and ``primitive root'' can be adapted. 
\par For a particular type of composite numbers, \smodtext leads to a multiplicative cyclic group.
As a result, new insights are gained for Sophie Germain pairs and twin primes.
\par The paper is organized as follows:
First, we recapitulate the standard knowledge in the context of mod. Then we define \smodtext and adapt well-known concepts to it.  
In particular, subsection~\ref{sec:artinsconjecture} adapts Artin's primitive root conjecture in the context of \smodtext to some composite numbers.    
Finally, section~\ref{sec:polynomials} gives a geometric interpretation of the congruence relation and closely related polynomials are constructed as an application.
\par Whereas Schick's sequences and the polynomials are defined only for odd numbers $n$, the new congruence relation $\smod$ may also be applied to even numbers.

\section{Preliminaries}
This section introduces the necessary notation (partially taken from Wikipedia \cite{WikMG}) and summarizes a number of well-known concepts from number theory.

\subsection{Multiplicative Group of Integers Modulo \emph{n}}
In number theory, the multiplicative group of integers modulo $n$ is well known and often described as follows.
\par The quotient ring $\mathbb{Z}/n\mathbb{Z}$ is defined by the following congruence relation on the ring of integers $\mathbb{Z}$:
\begin{equation}
\label{eq:modcong}
\begin{split}
    a &\equiv b \pmod n \\
    &\Updownarrow \\
    a - b &\in n\mathbb{Z} 
\end{split}
\end{equation}
where $n\mathbb{Z}$ is the ideal generated by $n$.
We denote the group of units (invertible elements) of $\mathbb{Z} /n \mathbb{Z}$ by $(\mathbb{Z} /n \mathbb{Z})^\times$.
For simplicity, we also refer to this group by $G_n$.
\par If $n$ is a power of an odd prime ($n=p^\alpha$ with $\alpha\in\mathbb{N}$), then there is the isomorphism 
\begin{equation*}
(\mathbb{Z} /p^\alpha \mathbb{Z})^\times \cong C_{\varphi(p^\alpha)},  
\end{equation*} 
where $C_m$ is the cyclic group of order $m$ and $\varphi$ is Euler's totient function.
\par In general, if $n$ is an odd composite number $n=p_1^{\alpha_1}p_2^{\alpha_2}\ldots p_l^{\alpha_l}$, then the group of units is isomorphic to the direct product of cyclic groups
\begin{equation*}
(\mathbb{Z} /n\mathbb{Z})^\times \cong (\mathbb{Z} /p_1^{\alpha_1} \mathbb{Z})^\times \times (\mathbb{Z} /p_2^{\alpha_2} \mathbb{Z})^\times\times\cdots\times (\mathbb{Z} /p_l^{\alpha_l} \mathbb{Z})^\times\cong C_{m_1} \times C_{m_2}\times\cdots\times C_{m_l} ,  
\end{equation*} 
where $m_i$ equals $\varphi(p_i^{\alpha_i})$.
\par The order $\lambda(n)$ of the largest cyclic subgroup of the group $G_n$ is given by Carmichael's function
\begin{equation*}
    \lambda(n)=\text{lcm}(m_1,m_2,\ldots,m_l).
\end{equation*}
This means that given $n$ and $a^{\lambda(n)}\equiv 1 \pmod n$ for any $a$ coprime to $n$, then $\lambda(n)$ is the smallest such exponent.  
\par The order of the group $G_n$ is $|G_n| = \varphi(n)$.
If $G_n$ is cyclic, its generators are called primitive roots modulo $n$.
Gauss \cite{Gau1801} showed that $G_n$ is cyclic (has primitive roots), if and only if $n$ is one of
\begin{equation*}
    n = 2, 4, p^\alpha~\text{or}~2p^\alpha,
\end {equation*} 
where $p$ is an odd prime and $\alpha$ a positive integer.

\subsection{Artin's Primitive Root Conjecture} 
In 1927, Artin formulated his primitive root conjecture \cite{Art65, Ramr05}.
It states that a given integer $a$ which is not a perfect square and not $-1$, 0 or 1, is a primitive root modulo infinitely many primes $p$.
If $N_a(x)$ denotes the number of such primes up to $x$ for a given integer $a$, he conjectured an asymptotic formula of the form
\begin{equation*}
N_a(x)\sim A_a\; \frac{x}{\ln x},
\end{equation*} 
as $x \to \infty$.
For the density of primitive roots $A_a$ he calculated $A_{Artin}\approx 0.3739$, independent of $a$.
\par It was later found that $A_a$ depends on $a$ (see e.g. Lenstra et al. \cite{Len14}) and it was proven by Heath-Brown \cite{Hea86} that one of $2, 3, 5$ is a primitive root modulo infinitely many primes.

\subsection{Cyclotomic Polynomials} \label{sec:cyclotomic}
The polynomial $x^n - 1$ with $n \in \mathbb{N}$ can be written as a product of
 so called cyclotomic polynomials, 
\begin{equation}
    \label{eq:prod}
    x^n - 1 = \prod_{d \mid n} \varPhi_d(x),
\end{equation}   
where $\varPhi_n(x)$ is the largest non-reducible polynomial factor of $x^n - 1$
 and is of degree $\varphi(n)$.
The M\"obius inversion formula directly leads to the expression
\begin{equation} \label{eq:minipoly}
    \varPhi_n(x) = \prod_{d \mid n}(x^d - 1)^{\mu(n/d)},
\end{equation}   
where $\mu$ is the M\"obius function.
\par An alternate definition of the cyclotomic polynomials is
\begin{equation*}
    \varPhi_n(x) = \prod_{\substack{k=1 \\ \gcd(n,k)=1}}^{n-1}(x - \xi^k),
\end{equation*}   
where $\xi^k\in \mathbb{C}$ are the roots of $x^n - 1=0$, i.e.\ the roots of unity   
\begin{equation} \label{eq:indexk}
    \xi^k=e^{2\pi ik/n}\textnormal{, where}~i^2=-1.
\end{equation} 
If $n$ is larger than 2, then they are palindromes, i.e.\ have symmetric coefficients.

\section{The Congruence Relation \smodtitle}

\subsection{Definition of \smodtitle}
In this subsection, the new congruence relation is defined.
Using this relation, the number of residue classes is halved compared to the canonical congruence relation defined in equation~\ref{eq:modcong}.
It is as versatile as mod for multiplication, but it destroys the additive structure of the ring $\mathbb Z$.
\begin{defn}
\label{defn:modstar}
Let $n$ be a positive natural number, and let $a, b \in \mathbb{Z}$ and coprime to $n$. 
Then we define a congruence relation with respect to multiplication as follows:
\begin{align*}
    a &\equiv b \smodn n \\
    &\Updownarrow \\
    a - b \in n\mathbb{Z}~&\textnormal{or}~a + b \in n\mathbb{Z}.
\end{align*}
If we define multiplication of congruence classes as $[a][b] = [ab]$, then we obtain the $\textbf{group G}_n^\star$.
\end{defn} 
Clearly, the relation in definition~\ref{defn:modstar} is an equivalence relation on the set of integers coprime to $n$: it is reflective, symmetric and transitive. Furthermore, it is compatible with multiplication and therefore a congruence relation. Note that, due to the loss of the additive structure, this does not define a new quotient ring similar to $\mathbb{Z}/n\mathbb{Z}$. Rather, it should be interpreted as a ``compression'' of the equivalence classes in $G_n$.
\par With each equivalence class in $G_n^\star$, we can associate a positive representative smaller than $n$.
These representatives form the reduced residue system \smodtext $n$. For example, if $n = 9$, we get the residue system $\{1, 5, 7 \}$ for \smodtext as opposed to $\{1, 2, 4, 5, 7, 8\}$ for mod.
The representative $\mathcal{R}(a)$ of an integer $a$ in the reduced residue system mod $n$, can easily be computed as 
\begin{equation}
    \label{eq:modsrepr}
    \mathcal{R}(a) = \begin{cases}
    \qquad a & \text{if $a$ is odd,} \\
    n - a & \text{if $a$ is even.}
  \end{cases}
\end{equation}
Of course, one could swap ``even'' and ``odd'' in the above to obtain only even representatives.
\par For computations, it is often useful to freely use representatives, and obtain the representative of choice in the final step.  
For example, one would naturally prefer 2 over $n - 2$.
\par A third option for the representatives might be  
\begin{equation}
    \mathcal{R}(a) = \begin{cases}
    \qquad a & \text{if $a < n/2$,} \\ % simplified instead of $a\le (n-1)/2$
    n - a & \text{if $a > n/2$.}
  \end{cases}
\end{equation}
It would be required for even $n$, because both $a$ and $n-a$ would be odd.
\par Due to the fact that numbers and their additive inverses are considered equivalent, we have the following equality regarding the order of $G_n^\star$:
\begin{equation*}
    \mid G_n^\star \mid = \frac{\mid G_n \mid}{2} = \frac{\varphi(n)}{2} 
\end{equation*}

\subsection{Comparison with the Standard Modulo}
Throughout the rest of this paper, the following lemma will be useful to answer questions about quadratic residues and primitive roots of $G_n^\star$, if $n$ is a power of an odd prime.
It is essentially a way of converting the congruence relation from definition~\ref{defn:modstar} to the canonical modular congruence relation.
\begin{lem}
    \label{lem:qrcongrelation}
    Let $n$ be a power of an odd prime and let $a, b \in \mathbb{Z}$ coprime to $n$. Then the following property holds:
    \begin{align*}
        a &\equiv b \smodn n \\
        &\Updownarrow \\
        a^2 &\equiv b^2 \pmod n
    \end{align*}
\end{lem}
\begin{proof}
Note that $a^2 - b^2 = (a - b)(a + b)$.
If $n$ is prime, the result follows from the zero product property in the field $\mathbb{Z}/n\mathbb{Z}$
 and definition~\ref{defn:modstar}.
If $n = p^\alpha$, with $\alpha > 1$ and $p$ an odd prime, the property would not hold if both $a - b$ and $a + b$ are multiples of a power of $p$ --- both are nonzero, if we exclude the trivial case $b=\pm a$, since they are coprime to $n$.
This implies that $p \mid (a - b)$ and $p \mid (a + b)$, thus $p \mid 2a$, which contradicts the assumption that $a$ is coprime to $n$.
\end{proof}
\par From the definition of $G_n^\star$, it can be seen that
\begin{equation*}
    G_n^\star \cong G_n / \langle -1 \rangle,
\end{equation*}
where $\langle -1 \rangle \subset G_n$ denotes the subgroup generated by $-1$.
\par Note that lemma~\ref{lem:qrcongrelation} is also applicable to even $n$ of the form $n=2^\alpha p^\beta$.

\subsection{Relations to Schick's Recurrence Relation}
The recurrence relation given in equation~\ref{eq:schickseq2} satisfies 
\begin{equation*}
    q_i \equiv 2^i \smodn n,
\end{equation*}
applicable only to odd $n$. 
Let $\langle 2 \rangle \subseteq G_n^\star$ denote the cyclic subgroup generated by $2$, then the representatives of the elements of this subgroup correspond to the sequence defined above.
Specifically, the sequence consists of the representatives from equation~\ref{eq:modsrepr} of $\langle 2 \rangle = \langle n - 2 \rangle$, ordered by increasing value of the exponent $i$.
Schick denotes the order of $\langle 2 \rangle$ using $\text{pes}(n)$. It is the period of the sequence in equation~\ref{eq:schickseq1} or~\ref{eq:schickseq2}. 
\par A generalization of Schick's recurrence relation is possible by the following definition.
\begin{defn}
Let $n$ be a nonnegative integer number and $g$ a positive integer less than and coprime to $n$.
Then identical sequences can be generated by
\begin{equation*}
    q_i \equiv g^i \smodn n
\end{equation*} 
or --- where the absolute value is taken $g - 1$ times --- by 
\begin{equation*}
    q_{i + 1} = | n - |n - \cdots  | n - g q_i | \cdots | | \textnormal{ with }q_0=1
\end{equation*} 
\end{defn}
The terms of this sequence correspond to the cyclic subgroup $\langle g \rangle \subseteq G_n^\star$.
An explicit, non recursive, form of any sequence of the above form can thus be obtained by using \smodtext. This allows, for example, fast calculation of the terms in such a sequence.

\subsection{Applying \smodtitle to Prime Powers} \label{subsec:primep}
As mentioned in the previous section, by applying \smodtext to an odd number $n$, the number of congruence classes is halved.
This leads to simplifications shown first for powers of an odd prime.
Begin by noting that the order of $G_n^\star$ with $n = p^\alpha$ is given by
\begin{equation*}
\mid G_n^\star \mid = \frac{\varphi(n)}{2} = \frac{(p - 1)~p^{\alpha - 1}}{2}.
\end{equation*} 
\par For a prime power $n$, $G_n$ is a cyclic group. Below, we show that $G_n^\star$ is also a cyclic group in this case.
\begin{theo}
Let $n = p^\alpha$ be the power of an odd prime, then $G_n^\star$ is a cyclic group of order $\lambda(n)/2 = \varphi(n)/2$.
\end{theo}
\begin{proof}
By the definition of the Carmichael function, we have for all $a \in \mathbb{Z}$ coprime to $n$ :
\begin{equation*}
    a^{\lambda(n)} \equiv 1 \pmod n
\end{equation*}
This is equivalent to (by lemma~\ref{lem:qrcongrelation}):
\begin{equation*}
    a^{\lambda(n)/2} \equiv 1 \smodn n,
\end{equation*}
where $\lambda(n) / 2$ is the smallest such exponent.
\end{proof}
\par Gauss showed that for $n = p^\alpha$, there are $\varphi(\varphi(n))$ primitive roots.
In the context of \smodtext,  $g$ is considered equivalent to its additive inverse $n - g$.
The number of primitive roots is thus $\varphi(\varphi(n)/2))$, the value of which depends on the parity of $\varphi(n)/2$. 
\begin{theo}
Let $n$ be the power of an odd prime ($n=p^\alpha$), then the average density of primitive roots in $G_n^\star$ is 50\% higher than in $G_n$.
\end{theo}
\begin{proof}
Recall that $|G_n|=\varphi(n)$ and $|G_n^*|=\varphi(n)/2$, thus 
\begin{equation*}
    \frac{\varphi(|G_n^\star|)}{|G_n^\star|} = 
    \begin{cases}   
        \;\;~\varphi(|G_n|)/|G_n| &\text{for $p=4k+1$},\\  
        2~\varphi(|G_n|)/|G_n| &\text{for $p=4k+3$}. 
    \end{cases}   
\end{equation*}
Assuming equal frequencies of the two forms of $p$ leads in the average to
\begin{equation*}
    \frac{\varphi(|G_n^\star|)}{|G_n^\star|} = f~\frac{\varphi(|G_n|)}{|G_n|}\text{ with } \bar{f}=1.5.  
\end{equation*}
\end{proof}
Artin's  primitive root conjecture may be adapted to $\smod$.
It then states that any integer $a>1$ which is not a perfect square, is a primitive root $\smod$ infinitely many primes $p$ and that the density of primitive roots converges to a constant as the number of such primes approaches infinity.
For $a=2$ in the context of $\smod$, Schick found a density of primitive roots $A_2\approx 0.561\approx 1.5\; A_{Artin}$, calculated for the primes up to 2,000,000.\cite{Sch08}

\subsection{Applying \smodtitle to Odd Composite Numbers}
This section discusses the structure of $G_n^\star$ for all odd composites $n$.
\begin{defn} \label{defn:j}
To distinguish between different cases we define the number
\begin{equation*}
    j(n) = \frac{\varphi(n)}{\lambda(n)}
      = \gcd(\varphi(p_1^{\alpha_1}), \ldots, \varphi(p_k^{\alpha_k}))
      \textnormal{~for}~n = \prod_{i = 1}^k p_i^{\alpha_i}.
\end{equation*}
This number is listed in OEIS as sequence A034380~\cite{OeisJ}.
\end{defn}
The case $j=1$ is found only for prime powers $n=p^k$.
It was shown in subsection~\ref{subsec:primep} that $G_n^\star$ is a cyclic
 group of order $\lambda(n)/2$.
\par Assume now that $n = p_1^{\alpha_1}p_2^{\alpha_2}$.
The lemma below shows that $G_n$ is a direct product of two cyclic groups.
\begin{lem} \label{lem:decompositionJ}
If $n = p_1^{\alpha_1}p_2^{\alpha_2}$ with $p_1$ and $p_2$ odd, and $j$ as in definition~\ref{defn:j}, then
\begin{equation} \label{eq:isomorphismJ}
    G_n \cong C_j \times C_{\lambda(n)}.
\end{equation}
\end{lem}
\begin{proof}
Shanks~\cite{ShanksNT} considers a special factorization $\phi_n$ of $\varphi(n)$:
\begin{equation*}
    \phi_n = p_1^{\alpha_1 - 1} p_2^{\alpha_2 - 1}
     \left( \prod_k q_{1, k}^{\beta_{1, k}}\right)
     \left( \prod_k q_{2, k}^{\beta_{2, k}}\right),
\end{equation*}
where each of the powers is written in expanded form (e.g.\ $3^2 = 9$).
For more detail, see Chapter 2, \S34 of~\cite{ShanksNT}.
This factorization can be used to decompose $G_n$ as follows:
\begin{equation*}
    G_n \cong C_{f_1} \times C_{f_2} \times \cdots \times C_{f_r}
    \text{~with}~f_1 \le f_2 \le \cdots \le f_r,
\end{equation*}
It is shown that the product of the largest power of each distinct prime in $\phi_n$
 is $f_r$. 
To obtain $f_{r - 1}$, apply the same procedure to $\phi_n/f_r$.
It is not difficult to see that $f_r = \lambda(n)$.
Since $\phi_n/f_r = j$ contains every prime power at most once, the procedure
 stops after the second step.
Hence, $r = 2$ and equation~\ref{eq:isomorphismJ} follows.
\end{proof}
The case $j = 2$ occurs only for $n$ of the same form as in the above lemma
 and $\gcd(\varphi(p_1^{\alpha_1}),\varphi(p_2^{\alpha_2}))=2$.
Since $C_2 \cong \langle -1 \rangle$, it follows from the isomorphism in
 equation~\ref{eq:isomorphismJ} that $G_n^\star$ is a cyclic
 group of order $\varphi(n)/2$.
\par For $j\ge4$ a unique cyclic group of order $\lambda(n)$ can be found in the following case. Let $n$ be of the same form as in the above lemma. If $p_1<p_2$ and $p_1^k$ (with $1\le k<\alpha_1$) is the largest power of $p_1$ dividing $p_2-1$, then one gets $j=2p_1^k$ and (\smodtext $p_1^{\alpha_1-k}p_2^{\alpha_2})$ returns the unique cyclic group. Examples are $n=63, 189, 275$. 

\subsection{Cyclicity of $G_n^\star$}
The well known theorem of Gauss \cite{Gau1801} that $G_n$ is cyclic exactly for the four forms $n=2,4,p^\alpha,2p^\alpha$ (with $p$ an odd prime and $\alpha$ a natural number),  shall now be adapted to \smodtext.
\begin{theo}
Let $G_n^\star$ be defined as above (definition~\ref{defn:modstar}) if and only if $n$ is one of the following:
\begin{equation*}
    n = \begin{cases} 
        p^\alpha~\text{or}\\
        p^{\alpha}q^{\beta}~\text{with}~\gcd\left(\varphi(p^\alpha),\varphi( q^\beta) \right)  = 2,
    \end{cases}
\end{equation*}
with $p$ and $q$ distinct odd primes and $\alpha$ and $\beta$ positive integers.
\end{theo}
\begin{proof}
It was shown in the previous section that if $j=1$ and $j=2$ (definition~\ref{defn:j}),
 then $G_n^\star$ is cyclic.
Assume that $n$ is not of the above form, then $j > 2$ or equivalently
 $\varphi(n)/2 > \lambda(n)$.
In this case, $G_n^\star$ cannot be cyclic because no element is of order $\varphi(n) / 2$.
\end{proof}
\begin{defn}
Let $p$ and $q$ be distinct odd primes and $\alpha$ and $\beta$ positive integers with $\gcd((p-1)p^{\alpha-1},(q-1)q^{\beta-1})=2$. Then we denote the product of odd primes $n = p^\alpha q^\beta$  a $\textbf{cyclic semiprime}$.
\end{defn}
Ki-Suk Lee et al.\cite{SemiPrimitiveRoots, MultGroupsSemiPrimitive} introduce
 ``good semi-primitive roots''.
These correspond to the generators of $G_n^\star$, i.e.\ primitive roots \smodtext $n$.
Good semi-primitive roots are also defined for even $n$.
For all odd $n$, there is the isomorphism $G_{2n} \cong G_n$.
This allows defining $G_{2n}^\star$ in terms of $G_n^\star$.

\subsection{Cyclic Semiprimes and Artin's Conjecture}
\label{sec:artinsconjecture}
The product of two twin primes, of a pair of Sophie Germain primes and of many other pairs of primes are cyclic semiprimes.
\par We demonstrate the case of Sophie Germain prime pairs: it is well known that they are of the form $p_1=6k-1$ and $p_2=12k-1$ with $k \in \mathbb{N}$. (To eliminate in advance all $p_i$ divisible by 5, one can additionally require that $k \equiv 0,2 \text{ or } 4\pmod 5$ and so on for $7, 11\dots$)  
\par Because the group $G_n^\star$ for $n_{SG}=(6k-1)(12k-1)$ is cyclic in the context of \smodtext, we can adapt Artin's primitive root conjecture and state --- if it holds --- that a given prime $b$ is a primitive root \smodtext infinitely many cyclic semiprimes $n_{SG}$ of Sophie Germain pairs and that the density of primitive roots approaches a constant value as the number of such pairs approaches infinity. 
\par If $N_b(x)$ denotes the number of Sophie Germain primes less than $x$ for which $b$ is a primitive root $(\smod n_{SG})$, then an asymptotic formula is conjectured of the form 
\begin{equation*}
N_b(x)\sim A_b\;\int_2^x \! \frac{\text{d}x}{\ln(x)\ln(2x+1)},
\end{equation*} 
as $x$ approaches infinity. 
\par Heuristically, the density of primitive roots $A_b$ was calculated  to be in the interval (0.28, 0.47) for $x=10,000,000$ and primes $b<20$.

\subsection{Quadratic Residues and their Roots}
A difficult problem in number theory, is to find the root of a quadratic residue. Lagrange and Legendre found solutions for specific cases. In general, one has to use algorithms, such as that of M\"uller \cite{Muel04} or Tonelli-Shanks \cite{WikQuadRoot}. Using \smodtext, we found a closed-form solution one ``level'' higher than using the standard modulo. 
\par Level 1: Let $n=p^\alpha$ be an odd prime power with $\varphi(n)/2$ odd. Then every element $b\in G_n^\star$ is a quadratic residue, and 
\begin{equation} \label{eq:root}
x \equiv b^{(\varphi(n)/2+1)/2}\smodn n
\end{equation}
is a root of $b$, a solution of the equation $x^2\equiv b\smodn n$. To prove equation~\ref{eq:root} square it and remember, that $\varphi(n)/2$ is the size of the cyclic group $G_n^\star$. $x$ is itself a quadratic residue. (In the context of the standard modulo the second square root of $b$ would be $n-x$.)  

\par Level 2: Let $n=p^\alpha$ be an odd prime power or $n=p^\alpha q^\beta$ a cyclic semiprime with $j=2$,  with in either case  $\varphi(n)/4$ odd. Then $G_n^\star$ contains the subset (50\%) of all quadratic residues and the coset (50\%) of the quadratic non-residues. The subset is a cyclic group of size $\varphi(n)/4$. Let $b$ be a quadratic residue, then 
\begin{equation} \label{eq:root2}
x \equiv b^{(\varphi(n)/4+1)/2}\smodn n
\end{equation}
is a root of $b$. The proof is similar as above for level 1. $x$ is itself a quadratic residue. Multiplying the set of the quadratic residues by a primitive root yields the coset, which is not a cyclic group and which contains all primitive roots and a second root of $b$. (Equation~\ref{eq:root2} leads with the standard mod only for a few $b$-values to the correct answer.) Examples for level 2 are $n=13,77,605$ and some products of Sophie Germain pairs.

\par Level 3: Let $n=p^\alpha$ be an odd prime power or $n=p^\alpha q^\beta$ a cyclic semiprime with $j=2$,  with in either case $\varphi(n)/8$ odd. Then $G_n^\star$ contains a subset (25\%) of all biquadratic residues, the coset (25\%) of the pure quadratic residues, and the cocoset (50\%) of all quadratic non-residues. The biquadratic residues are a cyclic group of size $\varphi(n)/8$. Let $b$ be a biquadratic residue, then 
\begin{equation*}
x \equiv b^{(\varphi(n)/8+1)/2}\smodn n
\end{equation*}
is a root of $b$. $x$ is itself a biquadratic residue.  Multiplying the subset of the biquadratic residues by an appropriate element of the pure quadratic residues yields the coset, and multiplying the subset united with the coset by a primitive root yields the cocoset.  Example for level 3 are $n=41,143$.

\par Level 3 yields interesting results. Each squaring halves the number of elements. With the standard modulo, the first squaring divides the number of elements by 4 (a combination of uniting $a$ with $n-a$, $b$ with $n-b$ and squaring $a$ and $b$). But the overall picture is similar. Therefore, the real advantage of \smodtext lays in level 2.
\par One could save one loop in the Tonelli-Shanks algorithm \cite{WikQuadRoot} by using \smodtext, but the starting quadratic residue would have to be odd. It could be an advantage e.g.\ in the quadratic sieve algorithm \cite{Pom88}.

\subsection{Generalized Primitive Roots and \smodtitle}
Li and Pomerance \cite{Li02} generalize the term primitive root to arbitrary moduli and study their density. If \smodtext is applied in this context, one finds at most half as many generalized primitive roots. 
\par Our paper is focused on odd numbers. However, \smodtext may also be applied to even numbers. (Note, the representatives from equation~\ref{eq:modsrepr} have to be chosen differently, e.g. $<n/2$, instead of odd or even.)
If $g$ is a generalized primitive root, i.e.\ has order $\lambda(n)$, then $n-g$ is also a generalized primitive root and \smodtext unites $g$ and $n-g$.
Additional generalized primitive roots may disappear, if $\lambda(n)/2$ is odd.
Specifically, for the ratio $r$ of generalized primitive roots mod $n$ to generalized primitive roots \smodtext $n$, one finds the following.
\par Let $n$ be a positive composite number (even or odd) of the form $n=p_1^{\alpha_1}p_2^{\alpha_2}p_3^{\alpha_3}\ldots p_l^{\alpha_l}$ with each $\varphi(p_i^{\alpha_i})/2$ an odd number larger than 1, then those generalized primitive roots $g$ with $g^{\lambda(n)/2}\equiv -1 \pmod n$ have the halved order $\lambda(n)/2$  (\smodtext $n$) and are no longer generalized primitive roots. 
\par In the terminology of Li and Pomerance (\cite{Li02}, page 3), one has $p=2$ and $\nu_2=l$ and finds the ratio
\begin{equation*}
     r = \begin{cases}
            2 &\text{if $\lambda(n)/2$ is even}\\
           2 \frac{1}{2^{l-1}-1}     &\text{otherwise} % How to write "2 + fraction" (in German we say "gemischter Bruch") to be understood? 
     \end{cases}
\end{equation*}
The second case with ratio $r=2\frac{1}{3}$ occurs e.g.\ for $n=84$ and $n=231$.

\section{Polynomials Related to Odd Integers}
\label{sec:polynomials}

\subsection{Definition of Polynomials Based on Chords}
Three distinct polynomials are defined: the polynomials $S_k(s)$ relating other chords to an arbitrarily selected first one $s$, the polynomials $P_m(s)$ comprising all chords related to an odd number $n=2m+1$, and the polynomials $\varPsi_n(s)$, the largest irreducible factor of $P_m(s)$.
\paragraph{The polynomials $\textbf{S}_k(s)$}
We start by using the symmetry of the cyclotomic polynomial (see subsection~\ref{sec:cyclotomic}) to combine two variables into one, halving the degree.
\begin{defn}
Let $n$ be an odd positive integer and $x$ a point on the unit circle in the complex plane.
Then $x$ and $x^{-1}$ are complex conjugates and we define the real variable 
\begin{equation*}
    s = x + x^{-1},
\end{equation*}
and the polynomial $S_k(s)$ for powers of $x$ as  
\begin{equation*} \label{eq:subst}
    S_k(s)=x^k+x^{-k},
\end{equation*}
closely related to the Chebyshev polynomials of the first kind $T_k$
 (found e.g.\ in~\cite{WolfCheby}), namely
\begin{equation*}
    S_k(s)=2 \; T_k(s/2).
\end{equation*}
\end{defn}  
Note that $S_k$ is a chord of the unit circle.
The first few chords are 
\begin{equation*}
  S_0=2,\quad S_1=s,\quad S_2=s^2-2,\quad S_3=s^3-3s,
\end{equation*}
with the recurrence relation
\begin{equation*}
  S_n=s\cdot S_{n-1}-S_{n-2}\quad\textnormal{or in general}\quad S_n=S_k\cdot S_{n-k}-S_{n-2k} 
\end{equation*}
and the explicit formula
\begin{equation*}
    S_k(s) =\frac{(s+\sqrt{s^2-4})^k+(s-\sqrt{s^2-4})^k}{2^k}
\end{equation*}
\par The polynomials exhibit a weighted orthogonality
\begin{equation*}
    \int_{-2}^{2} \frac{S_k(s)S_l(s)}{\sqrt{4-s^2}}ds = \begin{cases}
    0 & \text{if }k\not =l,\\
    2\pi & \text{if }k=l \not =0,\\
    4\pi & \text{if $k=l=0$,}
    \end{cases}
\end{equation*}
and a nesting property
\begin{equation*}
      S_k\big(S_l(s)\big)=S_{kl}(s). 
\end{equation*}
The polynomials $S_k(s)$ are identical to the Dickson polynomials of the first kind $D_n(x,\alpha)$ with $\alpha=1$ introduced by L. E. Dickson in 1897  (found e.g.\ in~\cite{WikDick}).
\paragraph{The polynomials $\textbf{P}_m(s)$}
Consider the polynomial $x^n - 1$ and factor out the real root $x = 1$ to get 
$(x^n-1)/(x-1)=1+\sum_{k=1}^{n-1}x^k$.
If $x$ is a root of unity, the sum is the well known Gauss sum.
Below, we define a polynomial $P_m$ based on this sum.
\begin{defn} \label{defn:pm}
Let $n$ be an odd integer.
By replacing $x^k + x^{-k}$ with $S_k(s)$ in the Gauss sum, a polynomial $P_m$
 is obtained:
\begin{equation*}
    P_m(s) = 1 + \sum_{k=1}^{m}S_k(s)\textnormal{, where}~m=\frac{n-1}{2}. 
\end{equation*}
\end{defn}
The first four polynomials are
\begin{equation*}
    P_0=1,\quad P_1=s+1,\quad P_2=s^2+s-1,\quad P_3=s^3+s^2-2s-1,
\end{equation*}   
with the recurrence relation 
\begin{equation*}
    P_m=s\cdot P_{m-1}-P_{m-2}\quad\textnormal{or in general}\quad P_m=S_k\cdot P_{m-k}-P_{m-2k},
\end{equation*}
and the explicit formula
\begin{equation} 
   \label{eq:explicit}
    P_m=\sum_{k=0}^{m}(-1)^i \left (\begin{matrix}i+k \\k\end{matrix}\right )
    s^k\text{, where}~i=\left\lfloor\frac{m-k}{2}\right\rfloor.
\end{equation}
From definition~\ref{defn:pm}, it follows that
\begin{equation*}
    P_{m}(\xi^k + \xi^{-k}) = P_m(2\cos(2\pi k/n)) = 0.
\end{equation*}
\paragraph{The polynomials $\varPsi_n(s)$}
The largest irreducible factor of $P_m$ will be the minimal
 polynomial of $2\cos(2\pi k/n)$ with $k$ coprime to $n$, i.e.\ the primitive
 roots of unity.
This polynomial will be denoted by $\varPsi_n(s)$.
It is clear that the polynomial $P_m$ equals the product of the minimal polynomials
 of the divisors of $n$.
This leads to the theorem below.
\begin{theo} \label{theo:minimalPolynomial}
Let $n$ be an odd integer, then the minimal polynomial $\varPsi_n$ is
\begin{equation}
    \varPsi_n(s)=\prod_{d \mid n} \left(P_{(d-1)/2}(s)\right)^{\mu(n/d)},
\end{equation}
of degree $\varphi(n)/2$.
\end{theo}
\begin{proof}
This follows directly by applying the M\"obius inversion formula to
\begin{equation*}
    P_{(n - 1) / 2}(s) = \prod_{d \mid n} \varPsi_d(s).
\end{equation*}
Since $P_{(n - 1)/2}$ is of degree $(n - 1) / 2$, the degree of $\varPsi_n$
 can be computed as the inverse-M\"obius transform of the following well-known
 identity:
\begin{equation*}
    \frac{n - 1}{2} = \sum_{d \mid n} \frac{\varphi(d)}{2}
    \Rightarrow
    \frac{\varphi(n)}{2} = \sum_{d \mid n} \mu\left(\frac{n}{d}\right)\frac{d - 1}{2}
    = \deg{\varPsi_n}.
\end{equation*}
\end{proof}
D. Surowski and P. McCombs~\cite{SurowskiMinPoly} give a different, more complex,
 formula for the minimal polynomial of $2\cos(2\pi/p)$ for an odd prime $p$.
One can verify that their expression yields the same polynomials as the the explicit
 formula for $\varPsi_p(s) = P_m(s)$ from equation~\ref{eq:explicit}.
Note the similarity between theorem~\ref{theo:minimalPolynomial} and
 equation~\ref{eq:minipoly}.
Both are of course closely related, since we have the following relation between
    $\varPsi_n(s)$ and $\varPhi_n(x)$:
\begin{equation} \label{eq:cyclotomic}
    \varPhi_n(x) = x^{\varphi(n)/2}\cdot\varPsi_n\left(x + \frac{1}{x}\right).
\end{equation}

\subsection{The Roots}
We have seen that the variable $s$ can be understood as a chord of the unit circle.
The roots of $P_{(n-1)/2}(s)$ can be understood as the diagonals
 of a regular $2n$-gon inscribed in the unit circle.
We denote the roots by $\sigma$.
They may have a positive or a negative sign.
Each root (including the sign) appears exactly twice in the set
\begin{equation*}
    \{2\cos(2\pi k / n) \text{ with } k = 1, 2, 3, \ldots, n - 1\}.
\end{equation*}
\par We have many choices in selecting a subset of representatives for $k$, and we are completely free in ordering and numbering the representatives.
Our preferred index transformation with respect to the numbering of $k={1,2,3,\ldots,n-1}$ is $j=|n-4k|$.
It returns only odd indices and begins at $j=1$ for $k=1$ by the shortest chord in ascending order.
Another choice is $j=n-|n-4k|$, yielding only even indices and beginning at $j=n-1$ for $k=1$ by the longest chord in descending order.
The third line in the next definition indicates that combinations of odd and even indices are also possible.

\begin{defn} \label{def:altexpression}
Let $n$ be a positive odd integer.
Then a complete representative set of chords of size $(n-1)/2$ related to
 the number $n$ is defined by one of the three lines
\begin{equation*} \label{eq:altexpression}
    \sigma_j = \begin{cases}       
        (-1)^{(n-j)/2}\;2\sin\Big(\frac{\pi}{2}\frac{j}{n}\Big)
        &\textnormal{if}~j\in\{1,3,\ldots,n-2(,n)\},\\
        (-1)^{j/2}\;2\cos\Big(\frac{\pi}{2}\frac{j}{n}\Big)
        & \textnormal{if}~j\in\{(0,)2,4,\ldots,n-1\}, \\
        mixed
        &\textnormal{if}~j\in\{(0,)1,2, \ldots,\frac{n-1}{2}\}.
    \end{cases}
\end{equation*}
The zero element, which is not a root of a polynomial, has been included
 in round brackets as it will prove to be useful in chord arithmetics ($\sigma_0=\sigma_n=2$).  
\end{defn} 
\par The periodicity of trigonometric functions and the symmetry between sine and cosine --- mirrored at $\pi/4$ --- leads to the congruence relation of this paper. Therefore, chords are related as follows:
\begin{align*}
    j &\equiv i \smodn n \\
    &\Updownarrow \\
    \sigma_j &= \sigma_i.
\end{align*}
Note, so far we have only considered the polynomials $P_{(n-1)/2}(s)$, where the index $j$ need not be coprime to $n$. The above holds nevertheless, but the set of the related congruence classes with multiplication is not a group. If solely the roots $\sigma_j$ with $j$ coprime to $n$ are considered, one gets to the minimal polynomials $\varPsi_n(s)$  and can write
\begin{equation*}
    \varPsi_n(s)=\prod_{\substack{j=1 \\ \gcd(n,j)=1}}^{(n-1)/2}(s-\sigma_j).
\end{equation*}   
All the polynomials over the integers considered in this section --- $x^n-1$,\\*$\varPhi_n(x), P_{(n-1)/2}(s)$ and $\varPsi_n(s)$ --- are monic and the value of their last coefficient is $\pm1$.
Because this constant is the product of the polynomial's roots, we get
\begin{equation*}
    \Bigg|\prod_{all\; roots}\sigma_j\Bigg|=1,
\end{equation*}  
which holds for the roots of $\varPsi_n(s)$ and $P_{(n-1)/2}(s)$.
\par A resume: the roots of the polynomials $x^n-1$ and $\varPhi_n(x)$ are roots of unity, they are all but one complex and $n$ may be even. The roots of the polynomials $P_{(n-1)/2}(s)$ and $\varPsi_n(s)$ are chords of the $2n$-gon, are real and $n$ must be odd. The polynomials $\varPhi_n(x)$ and $\varPsi_n(s)$ are the minimal polynomials in each case.

\subsection{Chord Arithmetic}
The chords $\sigma_j$ from definition~\ref{eq:altexpression} combine sine and cosine functions, use their periodicity and symmetry, and map these properties to the index number $j$. This subsection demonstrates that the arithmetic of chords becomes an arithmetic of index numbers.
\par Let $n$ be an odd positive integer and $\sigma_i$ and $\sigma_j$ chords related to $n$, then the following equation holds:
\begin{equation}
    \label{eq:chordproduct}
    \sigma_i \sigma_j=\sigma_{i+j}+\sigma_{i-j},
\end{equation}  
where the index numbers $i+j$ and $i-j$ are the residues (\smodtext $n$). In the case $i=j$ an additional chord is found  $\sigma_0=\sigma_n=2$, the diameter of the unit circle.
\par Equation~\ref{eq:chordproduct} follows readily from the equation
\begin{equation*}
2\sin(\alpha)\sin(\beta)=\sin(\alpha+\beta)+\sin(\alpha-\beta).
\end{equation*}   

\paragraph{Examples.}
The use of the above rule is demonstrated for two examples, $P_3$ and $P_6$.
In accordance with the fundamental theorem of algebra, all coefficients of the polynomial $P_m$, except the first one, are composed of products and sums of chords.
Products of chords can be transformed into sums of chords by equation~\ref{eq:chordproduct}:
\begin{equation*}
      P_3=\prod_{j=1}^{3}(s-\sigma_j)=s^3-(\sigma_1+\sigma_2+\sigma_3)s^2+(\sigma_1\sigma_2+\sigma_1\sigma_3+\sigma_2\sigma_3)s-\sigma_1\sigma_2\sigma_3
\end{equation*}   
For the Gauss sum $\sigma_1+\sigma_2+\sigma_3$ the result is known, it is $-1$.
Applying~\ref{eq:chordproduct} to the third term and choosing the appropriate representative $j \in \{1,2,3\}$, yields: 
\begin{equation*}
       \sigma_1\sigma_2+\sigma_1\sigma_3+\sigma_2\sigma_3=\sigma_3+\sigma_1+\sigma_4+\sigma_2+\sigma_5+\sigma_1 = 2(\sigma_1+\sigma_2+\sigma_3)=-2.
\end{equation*}   
To the last term $\sigma_1\sigma_2\sigma_3$, equation~\ref{eq:chordproduct} is applied sequentially.
\begin{equation*}
      \sigma_1\sigma_2\sigma_3=(\sigma_3+\sigma_1)\sigma_3=\sigma_6+\sigma_0+\sigma_4+\sigma_2 = \sigma_1+2+\sigma_3+\sigma_2=1.
\end{equation*}   
Substituting these results into the expression above for the polynomial yields
\begin{equation*}
P_3=s^3+s^2-2s+1.
\end{equation*}   
\par Polynomials $P_m$ for primes or prime powers of the form $n=4k+1$ can be factored as demonstrated here for $n=13$: 
\begin{equation*}
     P_6=s^6+s^5-5s^4-4s^3+6s^2+3s-1=(s^3+c_1 s^2-s-1-c_1)(s^3+c_2 s^2-s-1-c_2),
\end{equation*}
where $c_1=\sigma_1+\sigma_3+\sigma_9=\frac{1-\sqrt{13}}{2}$ and $c_2=\sigma_5+\sigma_7+\sigma_{11}=\frac{1+\sqrt{13}}{2}$. 
\par In this example the representatives are chosen differently, namely $j\in\{1,3,5,7,9,11\}$. The numbers 1, 3, and 9 are the primitive roots in $G_n^\star$. The value of $c_1$ can be deduced from the Gauss sum \cite{Gau1801} of all quadratic residues in $G_n$, that is $\sqrt{13}$.
\par The following more general formula may be used to transform a product of chords into a sum of chords. Let $n$ be an odd positive integer and $\sigma_{j_k}$ with $k=\{1,2,3,\ldots,m\}$ a bunch of not necessarily distinct chords related to $n$. Then the formula is 
\begin{equation*}
     \prod_{k=1}^{m}\sigma_{j_k}= \sum_{l=1}^{2^{m-1}}\sigma_{i_l},
\end{equation*}  
where the index numbers $i_l$ are the $2^{m-1}$ distinct combinations of $\pm$ signs in the next line
\begin{equation*}
    i_l\equiv j_1\pm j_2\pm j_3\pm\ldots\pm j_m\; (\smod n).
\end{equation*} 
Clearly, on the right side of the equation there are more terms than on the left side and one can expect, that many $i_l$ values are identical.  

\subsection{Geometric Interpretation of the Chords} % The chords are only defined for odd $n$
\begin{figure}[h!]
    \centering
    \includegraphics[trim=0 75mm 0 70mm, clip=true,width=\textwidth]{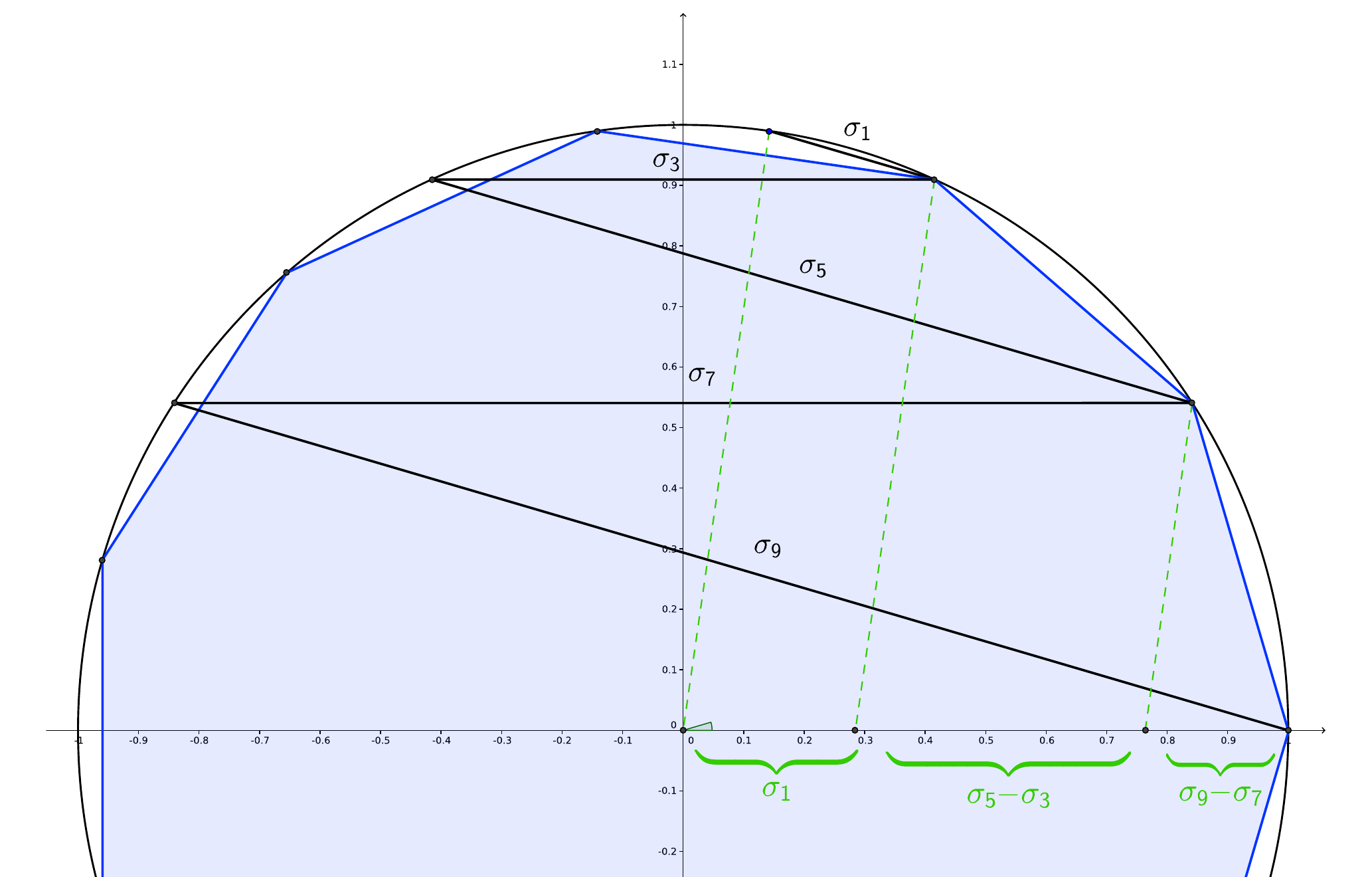}
    \caption{Chords associated with $n = 11$.}
    \label{fig:chords11}
\end{figure}
Figure~\ref{fig:chords11} shows the upper half of a polygon with 11 corners inscribed in a unit circle.
Additionally, some corners of an 22-gon are marked.
Five chords $\sigma_1, \sigma_3, \sigma_5, \sigma_7$ and $\sigma_9$ are associated with the number 11.
They are diagonals or a side of the 22-gon and may be drawn at several positions.
Here they are drawn to demonstrate the sum rule.
The angles between two intersecting chords is the constant $\arcsin(\sigma_1)=\pi/n$.
The chords $\sigma_1, \sigma_5$ and $\sigma_9$ have a negative sign (definition~\ref{def:altexpression}).
The green lines demonstrate that $-\sigma_9+\sigma_7-\sigma_5+\sigma_3-\sigma_1=-1$, the Gauss sum.\\
The figure also demonstrates the following properties:
\begin{itemize}
    \item The equality $\sigma_j=\sigma_{n-j}$, by mirroring around the line $y = x$.
    \item The recurrence relation $\sigma_j=\sigma_1\cdot \sigma_{j-1}-\sigma_{j-2}$.
    \item The multiplication of chords (equation~\ref{eq:chordproduct}).
    \item Schick's geometric construction to get his original sequence (equation~\ref{eq:schickseq1}).
\end{itemize}

\section{Conclusions}
We have defined a new congruence relation that can be used to study the behavior of the sequence given by Schick and other, similar, sequences.
\par By defining \smodtext, we can simplify several aspects of Schick's work.
Furthermore, the multiplicative group of integers \smodtext $n$ has a number of properties that are interesting by themselves.
In particular, \smodtext yields relatively more quadratic residues and the process of finding square roots is simplified.
For some special composite numbers, the multiplicative group of integers \smodtext $n$ is cyclic.
In this case, one can define "primitive roots" and adapt Artin's primitive root conjecture.
Examples of such composites are Sophie Germain and twin prime pairs.
\par Finally, polynomials related to odd integers and \smodtext are introduced.
These lead to a simple expression for the minimal polynomial of $2\cos(2\pi/n)$, where $n$ is odd. 

\section{Acknowledgement}
Carl Schick has given us the basic ideas to start this work and he contributed in many discussions regarding its development. Kfir Barhum was a big help in writing a mathematically sound paper. Further, we have to thank Alan Szepieniec, Hans Bachofner, Hans Heiner Storrer, Fritz Gassmann and Juraj Hromkovic for reading a draft and giving us valuable advice in proceeding our work. Hans Heiner Storrer found two errors in the first version of this paper and made us attentive to the work of Ki-Suk Lee et al..

\nocite{*}
\bibliographystyle{plain}  
\bibliography{modified_congruence}

\end{document}